\documentclass[11pt,reqno]{amsart}

 \usepackage{color}
 \usepackage{amsmath}
 \usepackage{amsfonts}
 \usepackage{setspace}
 \usepackage[toc,page]{appendix}
 
\newtheorem{thm}{Theorem}[section]
\newtheorem{prop}[thm]{Proposition}
\newtheorem{lem}[thm]{Lemma}

\newtheorem{definition}[thm]{Definition}

\numberwithin{equation}{section}

\renewcommand{\Re}{\operatorname{Re}}

\newcommand{\rg}{\operatorname{rg}}

\newcommand{\R}{\mathbb{R}}
\newcommand{\C}{\mathbb{C}}
\newcommand{\Hb}{\overline{\mathbb{H}}}

\usepackage{kantlipsum}
\usepackage[pagebackref,colorlinks=true,urlcolor=blue,linkcolor=red,citecolor=red]{hyperref}
\allowdisplaybreaks

\title[Stability of self-similar solutions]{On the stability of self-similar solutions to nonlinear wave equations}

\author{Ovidiu Costin}
\address{Department of Mathematics, The Ohio State University, 231 W 18th Ave, Columbus, OH, 43220, USA}
\email{costin@math.ohio-state.edu}

\author{Roland Donninger}
\address{Rheinische Friedrich-Wilhelms-Universit\"at Bonn,
Mathematisches Institut, Endenicher Allee 60, D-53115 Bonn, Germany}
\email{donninge@math.uni-bonn.de}
\thanks{Roland Donninger is supported by a Sofja Kovalevskaja Award granted by 
the Alexander von Humboldt Foundation 
and the German Federal Ministry of Education and Research.\\Ovidiu Costin is partially supported by the NSF DMS Grant 1108794.}

\author{Irfan Glogi\'c}
\address{Department of Mathematics, The Ohio State University, 231 W 18th Ave, Columbus, OH, 43220, USA}
\email{glogic.1@osu.edu}
\author{Min Huang}
\address{Department of Mathematics, City University of Hong Kong, Tat Chee Avenue, Kowloon, Hong Kong.} 
\email{ mihuang@cityu.edu.hk}

\begin{document}
\maketitle

\begin{abstract}
We consider an explicit self-similar solution to an energy-supercritical
Yang-Mills equation and prove its mode stability.
Based on earlier work by one of the authors, we obtain a fully rigorous proof
of the \emph{nonlinear} stability of the self-similar blowup profile.
This is a large-data result for a supercritical wave equation.
Our method is broadly applicable and
provides a general approach to stability problems related to self-similar solutions
of nonlinear wave equations.
\end{abstract}

\section{Introduction}

\noindent 
The development of singularities in finite time is one of the most stunning
features of nonlinear evolution equations.
Singularity formation (or ``blowup'') of the solution signifies a dramatic change in
the behavior of the underlying model or even the complete breakdown of the mathematical
description.
On the level of a fundamental physical theory, blowup occurs in Einstein's
equation of general relativity to indicate the dynamical formation of a black hole.
However, a rigorous treatment of Einstein's equation in this context is hopeless at the
present stage of research.
Consequently, it is a reasonable strategy to resort to simpler toy models
that capture some of the features of the more complicated system.
Natural candidates in this respect are energy-supercritical nonlinear wave equations with
a geometric origin such as wave maps or Yang-Mills models.

The easiest way to demonstrate finite-time blowup in a given evolution equation
is to construct self-similar solutions.
In exceptional cases it is even possible to obtain closed-form expressions.
The relevance of such solutions depends on their stability. After all, one would
like to obtain information on the \emph{generic} behavior of the system.
However, already at the linear level the stability analysis 
of self-similar solutions to nonlinear wave equations is very challenging since one
is confronted with highly nonself-adjoint spectral problems.
Consequently, standard methods do not apply.
This fact poses a serious obstacle to any rigorous analysis of the blowup
dynamics.

In the present paper we develop a general approach which is capable of handling the difficult
nonself-adjoint spectral problems related to self-similar blowup.
For the sake of simplicity, however, we focus on the concrete example of an energy-supercritical
Yang-Mills equation that displays blowup via an explicitly known
self-similar solution.

\subsection{An energy-supercritical Yang-Mills model}
For $\mu\in \{0,1,2,\dots,5\}$ let 
$A_\mu: \R^{1,5}\to \mathfrak{so}(5)$ be a collection of five fields on ($1+5$)-dimensional
Minkowski space with values in the matrix Lie algebra of $\mathrm{SO}(5)$.
In other words, for fixed $\mu$ and $(t,x) \in \R^{1,5}$,
$A_\mu(t,x)$ is a skew-symmetric real $(5\times 5)$-matrix.
One sets
\[ F_{\mu \nu}:=\partial_\mu A_\nu-\partial_\nu A_\mu+[A_\mu,A_\nu] \]
and considers the action functional\footnote{Einstein's
summation convention is in force. Greek indices take the values $0$ to $5$ whereas
latin indices run from $1$ to $5$. Our convention for the Minkowski metric
is $\eta=\mathrm{diag}(-1,1,1,1,1,1)$.}
\begin{equation}
\label{eq:action}
 \int_{\R^{1,5}}\mathrm{tr}(F_{\mu\nu}F^{\mu\nu}). 
 \end{equation}
Formally, this is reminiscent of Maxwell's theory. However, the commutator in the definition
of $F_{\mu\nu}$ introduces a very natural nonlinearity.
In this sense, Yang-Mills theory can be viewed as a nonlinear
generalization of electrodynamics.
The Euler-Lagrange equations associated to the action \eqref{eq:action} are
\[ \partial_\mu F^{\mu\nu}+[A_\mu, F^{\mu\nu}] =0 \]
and the ansatz \cite{D82, B02}
\[ A_\mu^{jk}(t,x)=(\delta^k_\mu x^j-\delta^j_{\mu}x^k)\frac{\psi(t,|x|)}{|x|^2} \]
yields the scalar nonlinear wave equation
\begin{equation}
\label{eq:YM} \psi_{tt}- \psi_{rr}
-\frac{2}{r}\psi_r+\frac{3\psi(\psi+1)(\psi+2)}{r^2}=0, 
\end{equation}
$\psi=\psi(t,r)$, for the auxiliary function $\psi: \R \times [0,\infty) \to \R$.
Eq.~\eqref{eq:YM} has been proposed as a model for singularity formation
in Einstein's equation \cite{BizTab01, B02, BC05, GM07}.
In general, (classical) Yang-Mills fields attracted a lot of interest by both the physics and
mathematics communities, see e.g.~\cite{A79, EM82a, EM82b, KM95, KT99, BOS04, T05,
KS05, S07,CKM08, KST09, RR09, S10}. 

Eq.~\eqref{eq:YM} is energy-supercritical \cite{BizTab01} and
large-data solutions can develop singularities in finite time
as is evidenced by the existence of self-similar solutions
of the form $\psi(t,r)=f(\frac{r}{1-t})$, see \cite{CST98}.
Bizo\'n \cite{B02} found an explicit example of this kind given by
\[ \psi_0(t,r)=f_0(\tfrac{r}{1-t}),\qquad f_0(\rho)=-\frac{8\rho^2}{5+3\rho^2}. \]
Numerical investigations \cite{BizTab01, B02, BC05} yield strong evidence that the solution $\psi_0$
gives rise to a stable self-similar blowup mechanism.
Motivated by this, the second author \cite{Don14} developed a complete nonlinear stability theory for the
solution $\psi_0$, see also \cite{D11, DSA11, DonSch12, DonSch14} for other types of
nonlinear wave equations.
However, the results in \cite{Don14} are conditional in the sense that they 
depend on a spectral assumption which could not be verified
rigorously so far. It is the aim of the present paper to close this gap.

\subsection{The mode stability problem}
The first important step in a stability analysis of the solution $\psi_0$ is to rule
out unstable modes. To this end, one introduces \emph{similarity coordinates} \cite{BC05}
\[ \tau=-\log(1-t),\qquad \rho=\frac{r}{1-t}. \]
Eq.~\eqref{eq:YM} transforms into
\begin{equation}
\label{eq:phi} \phi_{\tau\tau}+\phi_\tau+2\rho \phi_{\tau\rho}-(1-\rho^2)(\phi_{\rho\rho}
+\tfrac{2}{\rho}\phi_\rho)+\frac{3\phi(\phi+1)(\phi+2)}{\rho^2}=0 
\end{equation}
where $\phi(\tau,\rho)=\psi(1-e^{-\tau},e^{-\tau}\rho)$.
Due to finite speed of propagation one is mainly interested in the behavior inside
the backward lightcone
of the singularity, which corresponds to the coordinate domain $\tau\geq 0$, $\rho \in [0,1]$.
Note that the self-similar solution is independent of $\tau$ and simply given by
$f_0(\rho)$.
Next, one inserts the \emph{mode ansatz} 
\[ \phi(\tau,\rho)=f_0(\rho)+e^{\lambda\tau}u_\lambda(\rho),\qquad \lambda\in \C \]
and linearizes in $u_\lambda$.
This yields the ODE spectral problem
\begin{equation}
\label{eq:spec} -(1-\rho^2)(u_\lambda''+\tfrac{2}{\rho}u_\lambda')
+2\lambda \rho u_\lambda'+\lambda(\lambda+1)u_\lambda+\frac{V(\rho)}{\rho^2}u_\lambda=0 
\end{equation}
for the function $u_\lambda$, where the potential $V$ is given by
\[ V(\rho)=6+18f_0(\rho)+9f_0(\rho)^2=6\frac{25-90\rho^2+33\rho^4}{(5+3\rho^2)^2}. \]
Observe that Eq.~\eqref{eq:spec} has a singular point at the lightcone $\rho=1$ which is a consequence
of the fact that lightcones are the characteristic surfaces of Eq.~\eqref{eq:YM}.

Admissible solutions of Eq.~\eqref{eq:spec} with 
$\Re\lambda \geq 0$ lead to instabilities of
$f_0$ at the linear level.
However, it is not entirely trivial to determine what ``admissible'' in this context means.
This question can in fact only be answered once one has a suitable well-posedness theory
for Eq.~\eqref{eq:phi}.
The necessary framework is developed in \cite{Don14}
 and it turns out that if $\Re\lambda\geq 0$,
only smooth solutions are admissible.
Consequently, a nonzero solution $u_\lambda\in C^\infty[0,1]$ of Eq.~\eqref{eq:spec}
with $\Re\lambda\geq 0$ is called an \emph{unstable mode}.
The corresponding $\lambda$ is called an \emph{(unstable) eigenvalue}.
As a matter of fact, there exists an unstable mode. The function
\[ u_1(\rho):=-\rho f_0'(\rho)=\frac{80\rho^2}{(5+3\rho^2)^2} \]
turns out to be a smooth solution of Eq.~\eqref{eq:spec} with $\lambda=1$, as one easily
checks.
However, this mode is not a ``real'' instability of the solution $f_0$ but rather 
a consequence
of the time translation symmetry of Eq.~\eqref{eq:YM}.
Indeed, the profile $f_0$ defines in fact a one-parameter family of blowup solutions
given by
\[ \psi^T(t,r)=f_0(\tfrac{r}{T-t}) \]
where $T>0$ is a free parameter.
By the chain rule it follows that
\[ \partial_T \psi^T(t,r)|_{T=1}=-\tfrac{r}{(1-t)^2}f_0'(\tfrac{r}{1-t})
=-e^{\tau}\rho f_0'(\rho) \]
solves the linearized equation.
These observations lead to the following definition.

\begin{definition}
The solution $\psi_0$ (or $f_0$) is said to be \emph{mode stable} if 
$u_1$ is the only unstable mode.
\end{definition}

\subsection{The main result}
With these preparations at hand we can formulate our main result.

\begin{thm}
\label{thm:main}
The self-similar solution $\psi_0$ is mode stable.
\end{thm}

The first result of this kind was proved very recently for a similar problem
related to the wave maps equation \cite{CosDonXia14}.
However, the method we develop here is different and much more effective.
As a consequence, the main argument fits on a few pages
and the method easily generalizes to other types of nonlinear wave equations.
In view of the fact that rigorous research on self-similar blowup in supercritical
wave equations was blocked for a long
time by the difficulties related to these spectral problems, we hope that our method will
trigger new developments in the field.
In this respect we also remark that
Theorem \ref{thm:main} in conjunction with the theory developed in \cite{Don14} 
yields a fully rigorous proof of stable self-similar
blowup dynamics for the Yang-Mills equation \eqref{eq:YM}.
The precise statement is given in \cite{Don14}, Theorem 1.3.
We emphasize that this is a large-data result for an energy-supercritical 
wave equation.

\section{Removal of the symmetry mode}
\noindent Although the eigenvalue $\lambda=1$ is not connected to a real instability of the 
solution $\psi_0$,
it is still inconvenient for the further analysis. 
Consequently, it is desirable to ``remove'' it.
This can be done by a suitable adaptation of a well-known procedure from supersymmetric
quantum mechanics which we recall here briefly.

\subsection{Interlude on SUSY quantum mechanics}
Consider the Schr\"odinger operator $H=-\partial_x^2+V$ on $L^2(\R)$
with some nice potential $V$
and suppose there exists a ground state $f_0 \in L^2(\R)\cap C^\infty(\R)$, i.e., $f_0''=Vf_0$.
Assume further that $f_0$ has no zeros.
Then one has the factorization
\[ -\partial_x^2 + V=\left (-\partial_x-\frac{f_0'}{f_0}\right )
\left (\partial_x-\frac{f_0'}{f_0} \right )=:Q^*Q. \]
By interchanging the order of this factorization, one defines the SUSY partner $\tilde H$
of $H$, i.e., $\tilde H:=QQ^*$. Explicitly, the SUSY partner is given by
\begin{align*} 
\tilde H&=\left (\partial_x-\frac{f_0'}{f_0}\right )\left (-\partial_x-\frac{f_0'}{f_0}\right )
=-\partial_x^2-V+2\frac{f_0'^2}{f_0^2}=:-\partial_x^2 + \tilde V
\end{align*}
where $\tilde V=-V+2\frac{f_0'^2}{f_0^2}$ is called the SUSY potential.
The point of all this is the following. Suppose $\lambda$ is an eigenvalue
of $H$, i.e., $Hf=Q^*Q f=\lambda f$ for some (nontrivial) $f$.
Applying $Q$ to this equation yields 
$QQ^* Q f=\lambda Q f$, i.e., $\tilde HQf=\lambda Qf$.
Thus, if $Qf\not=0$, i.e., if $f\notin \ker Q$, $\lambda$ is an eigenvalue of $\tilde H$
as well.
Obviously, we have $\ker Q=\langle f_0\rangle$ and thus, if $\lambda\not=0$ is an 
eigenvalue of $H$, then it is also an eigenvalue of $\tilde H$.
Moreover, $0$ is not an eigenvalue of $\tilde H$ for if this were the case,
we would have $QQ^* f=0$ for a nontrivial $f$, i.e., $f\in \ker Q^*$ or $Q^* f\in \ker Q$.
The former is impossible since $\ker Q^*=\langle \frac{1}{f_0}\rangle$ but $\frac{1}{f_0}
\notin L^2(\R)$.
The latter is impossible since $\rg Q^* \perp \ker Q$.
In summary, $\tilde H$ has the same set of eigenvalues as $H$ except for $0$.

\subsection{The supersymmetric problem}
Now we implement a version of this SUSY factorization trick for our problem.
Note that the Frobenius indices of Eq.~\eqref{eq:spec} at $\rho=0$
are $\{-3,2\}$ and at $\rho=1$ 
we have $\{0,1-\lambda\}$.
Suppose $u_\lambda$ is an unstable mode of Eq.~\eqref{eq:spec} and $\lambda\not=0$.
By definition, $u_\lambda\in C^\infty[0,1]$
and from Frobenius theory it follows that $|u_\lambda(\rho)|\simeq \rho^2$
as $\rho\to 0+$ as well as $|u_\lambda(\rho)|\simeq 1$ as $\rho\to 1-$.
We define a new function $v_\lambda$ by\footnote{Observe that this transformation
depends on $\lambda$. This is the reason why Eq.~\eqref{eq:spec} is not equivalent
to a standard self-adjoint Sturm-Liouville problem. 
What happens is the following.
Since $|u_\lambda(\rho)|\simeq 1$ as $\rho\to 1-$,
the corresponding $v_\lambda$ behaves like $|v_\lambda(\rho)|\simeq (1-\rho)^{\Re \lambda/2}$.
The Hilbert space in which the spectral problem for $v_\lambda$ 
is symmetric is $L_w^2(0,1)$
with the weight $w(\rho)=\frac{1}{(1-\rho^2)^2}$.
Thus, if $\Re\lambda\leq 1$, the admissible solution $v_\lambda$ 
does not belong to $L^2_w(0,1)$!
Consequently, for $\Re\lambda\leq 1$ 
the self-adjoint formulation does not yield any information.
This shows that the spectral problem \eqref{eq:spec} is truly nonself-adjoint in nature.
In particular, there can be nonreal eigenvalues.
For $\Re\lambda>1$, on the other hand, one can indeed use Sturm oscillation theory
to exclude eigenvalues.}
\[ u_\lambda(\rho)=\rho^{-1}(1-\rho^2)^{-\lambda/2}v_\lambda(\rho). \]
From Eq.~\eqref{eq:spec} it follows that $v_\lambda$ satisfies 
\begin{equation}
\label{eq:v} -v_\lambda''+\frac{V(\rho)}{\rho^2(1-\rho^2)}v_\lambda
=\frac{\lambda(2-\lambda)}{(1-\rho^2)^2}v_\lambda. 
\end{equation}
For $\lambda=1$ we have
\[ v_1(\rho)=\rho (1-\rho^2)^{\frac12}u_1(\rho)
=\frac{80\rho^3(1-\rho^2)^{\frac12}}
{(5+3\rho^2)^2}. \]
We rewrite Eq.~\eqref{eq:v} as
\[ -v_\lambda''+V_1 v_\lambda=\frac{\lambda(2-\lambda)-1}{(1-\rho^2)^2}v_\lambda \]
with
\[ V_1(\rho)=\frac{V(\rho)}{\rho^2(1-\rho^2)}-\frac{1}{(1-\rho^2)^2}. \]
Then we have $v_1''=V_1 v_1$ and thus, Eq.~\eqref{eq:v} may be factorized as
\[ (-\partial_\rho-\tfrac{v_1'}{v_1})(\partial_\rho-\tfrac{v_1'}{v_1})v_\lambda=
\frac{\lambda(2-\lambda)-1}{(1-\rho^2)^2}v_\lambda \]
or
\[ -(1-\rho^2)^2(\partial_\rho+\tfrac{v_1'}{v_1})
(\partial_\rho-\tfrac{v_1'}{v_1})v_\lambda=[\lambda(2-\lambda)-1]v_\lambda. \]
We set $\tilde v_\lambda=(\partial_\rho-\frac{v_1'}{v_1})v_\lambda$ and apply the operator
$\partial_\rho-\frac{v_1'}{v_1}$ 
to the equation which yields the supersymmetric problem
\begin{equation}
\label{eq:SUSY}
 -(\partial_\rho-\tfrac{v_1'}{v_1})[(1-\rho^2)^2(\partial_\rho+\tfrac{v_1'}{v_1})]\tilde v_\lambda
=[\lambda(2-\lambda)-1]\tilde v_\lambda. 
\end{equation}
Note the asymptotics
\begin{align*} 
\frac{v_1'}{v_1}(\rho)&=3\rho^{-1}+O(\rho)\qquad (\rho\to 0+) \\
\frac{v_1'}{v_1}(\rho)&\sim -\tfrac12 (1-\rho)^{-1}\qquad (\rho\to 1-).
\end{align*}
Consequently, from the representation $v_\lambda(\rho)=\rho^3 h_\lambda(\rho^2)$, where
$h_\lambda$ is analytic near $0$, 
we get $\tilde v_\lambda(\rho)=O(\rho^4)$ 
near $\rho=0$
and from $v_\lambda(\rho)\sim c(1-\rho)^{\lambda/2}$ we infer
$\tilde v_\lambda(\rho)\sim c(1-\rho)^{\lambda/2-1}$ near $\rho=1$ (unless $\lambda=1$).
Writing out Eq.~\eqref{eq:SUSY} explicitly yields
\begin{equation}
\label{eq:SUSYexp}
-(1-\rho^2)^2 \tilde v_\lambda''+4\rho(1-\rho^2)\tilde v_\lambda'
+\frac{(1-\rho^2)\tilde V(\rho)}{\rho^2}
\tilde v_\lambda=\lambda(2-\lambda)\tilde v_\lambda
\end{equation}
with the supersymmetric potential
\[ \tilde V(\rho)=20\frac{15-2\rho^2+3\rho^4}{(5+3\rho^2)^2}. \]
Setting $\tilde u_\lambda (\rho)=\rho^{-1}(1-\rho^2)^{1-\lambda/2}\tilde v_\lambda(\rho)$ we find
the equation
\begin{equation}
\label{eq:SUSYu}
-(1-\rho^2)(\tilde u_\lambda''+\tfrac{2}{\rho}\tilde u_\lambda')+2\lambda\rho \tilde u_\lambda'
+(\lambda^2+\lambda-2)\tilde u_\lambda+\frac{\tilde V(\rho)}{\rho^2}\tilde u_\lambda=0.
\end{equation}
Note that the Frobenius indices of Eq.~\eqref{eq:SUSYu} are $\{-4,3\}$ at $0$
and $\{0,1-\lambda\}$ at $\rho=1$.
With minor modifications the same procedure can be performed in the case $\lambda=0$.
As before, we say that $\lambda\in \C$ is an unstable eigenvalue of Eq.~\eqref{eq:SUSYu}
if $\Re\lambda\geq 0$ and there exists a nontrivial solution 
$\tilde u_\lambda \in C^\infty[0,1]$ of Eq.~\eqref{eq:SUSYu}.
In summary, we have proved the following result.

\begin{prop}
\label{prop:SUSY}
Let $\lambda\not=1$ be an unstable eigenvalue of Eq.~\eqref{eq:spec}.
Then $\lambda$ is an unstable
eigenvalue of Eq.~\eqref{eq:SUSYu}.
\end{prop}

\section{Absence of unstable eigenvalues for the supersymmetric problem}

\noindent In this section we exclude unstable eigenvalues of Eq.~\eqref{eq:SUSYu}.
Via Proposition \ref{prop:SUSY} this implies the main result
Theorem \ref{thm:main}.

\begin{thm}\label{main_theorem}
The supersymmetric problem Eq.~\eqref{eq:SUSYu} 
does not have unstable eigenvalues.
\end{thm}

The Frobenius indices of~\eqref{eq:SUSYu} at $0$ are $-4$ and 3, hence 
the solution analytic at 0 has the power series representation
\begin{equation}\label{power_series_at_0}
\sum_{n=0}^{\infty}a_n(\lambda)\rho^{2n+3},\quad a_0\neq0. 
\end{equation}
Note that $\lambda$ is an eigenvalue of~\eqref{eq:SUSYu} 
if and only if the radius of convergence of~\eqref{power_series_at_0} is 
greater than 1. Therefore, our aim is to prove that for any 
$\lambda$ in the closed right half-plane (which from now on we denote by 
$\Hb$),~\eqref{power_series_at_0} cannot be analytically extended 
through $\rho=1$.

By substituting~\eqref{power_series_at_0} into~\eqref{eq:SUSYu} we obtain a four term 
recurrence relation (with the initial condition $a_0=1$ and $a_n=0$ for $n<0$)
\begin{equation}\label{recurrence_for_an} 
p_3(n)a_{n+3}+p_2(n)a_{n+2}+p_1(n)a_{n+1}+p_0(n)a_n=0,
\end{equation}
where
\begin{align*}
p_3(n)&=-100n^2-950n-1950,\\
p_2(n)&=-20n^2+(100\lambda-150)n+25\lambda^2+375\lambda-370,\\
p_1(n)&=84n^2+(120\lambda+462)n+30\lambda^2+330\lambda+630,\\
p_0(n)&=36n^2+(36\lambda+126)n+9\lambda^2+63\lambda+90.
\end{align*}
One can check that  $a_n=(-3/5)^n$ is an exact solution to~\eqref{recurrence_for_an}, hence the
order of the recurrence~\eqref{recurrence_for_an} can be reduced by one through the substitution 
\begin{equation}\label{subs}
b_n=a_{n+1}+\tfrac{3}{5}\, a_n.
\end{equation}
This yields a three term recurrence relation for $b_n$
\begin{equation}\label{rec2}
q_2(n)b_{n+2}+q_1(n)b_{n+1}+q_0(n)b_n=0,
\end{equation}
where
\begin{align*} 
q_2(n)&=p_3(n),\\
q_1(n)&=p_2(n)-\tfrac{3}{5}\, p_3(n),\\
q_0(n)&=p_1(n)-\tfrac{3}{5}\, p_2(n)+\tfrac{9}{25}\, p_3(n).
\end{align*}
After substituting for $p_i(n)$ in the last three relations, dividing all of them by 5 and using the $q_i$ 
notation for the new coefficients, we get
\begin{align*}
q_2(n)&=-20n^2-190n-390,\\
q_1(n)&=8n^2+(20\lambda+84)n+5\lambda^2+75\lambda+160,\\
q_0(n)&=12n^2+(12\lambda+42)n+3\lambda^2+21\lambda+30.
\end{align*}
By letting $A_n=q_1(n)/q_2(n)$ and $B_n=q_0(n)/q_2(n)$,~\eqref{rec2} becomes equivalent to
\begin{equation}\label{recurrence_for_bn}
b_{n+2}+A_nb_{n+1}+B_nb_n=0,
\end{equation}
with the initial condition $b_{-2}=0$ and $b_{-1}=1$.

\begin{lem}
\label{limit_of_rn} Given $\lambda$ in the complex plane, either
\begin{equation}\label{limit1_for_rn}
\lim_{n\rightarrow \infty} \frac{b_{n+1}(\lambda)}{b_n(\lambda)} = 1,
\end{equation}
or
\begin{equation}\label{limit2_for_rn}
\lim_{n\rightarrow \infty} \frac{b_{n+1}(\lambda)}{b_n(\lambda)} = -\frac{3}{5}.
\end{equation}
\end{lem}

\begin{proof} 
Since $\lim_{n\rightarrow \infty} A_n(\lambda)=-2/5$ and $\lim_{n\rightarrow \infty} 
B_n(\lambda)=-3/5,$ the characteristic equation associated to \eqref{recurrence_for_bn} is
\begin{equation}\label{char_eq}
t^2-\tfrac{2}{5}t-\tfrac{3}{5}=0.
\end{equation}
As the solutions to~\eqref{char_eq} (1 and $-3/5$) have distinct moduli, 
by a theorem of Poincar\'{e} (see, for example, \cite{Elaydi05}, p. 343, or 
\cite{Buslaev05}), either $b_n$ is zero eventually in $n$, 
or $\lim_{n\rightarrow \infty} b_{n+1}(\lambda)/b_n(\lambda)$ exists and it is 
equal to either 1 or $-3/5$. Now, for a fixed $\lambda$, $b_n$ cannot be zero eventually in $n$, 
since by backward induction from~\eqref{recurrence_for_bn} one would get $b_{-1}=0$, 
hence the claim follows.
\end{proof}

Note that in order to prove Theorem \ref{main_theorem}, it suffices to show that~\eqref{limit1_for_rn} 
holds for all $\lambda$ in $\Hb$, for that implies non-analyticity of 
~\eqref{power_series_at_0} at 1. Indeed, defining $f_\lambda$ by~\eqref{power_series_at_0} and 
$g_\lambda$ by  $g_\lambda(\rho)=a_0(\lambda)\rho+\sum_{n=0}^{\infty}b_n(\lambda)\rho^{2n+3}$, 
one easily checks that 
\begin{equation}\label{f_and_g}
f_\lambda(\rho)=\frac{5\rho^2}{3\rho^2+5}g_\lambda(\rho).
\end{equation}
So if~\eqref{limit1_for_rn} holds and therefore $g_\lambda$ is singular at $1$, then, by~\eqref{f_and_g}, 
so is $f_\lambda$. \par
Let $r_n=b_{n+1}/b_n$. Then from~\eqref{recurrence_for_bn} we obtain
\begin{equation}\label{recurrence_for_rn} 
r_{n+1}=-A_n-\frac{B_n}{r_n},
\end{equation}
where 
\begin{equation}\label{r_-1}
r_{-1}=\frac{b_0}{b_{-1}}= -A_{-2}(\lambda)= \frac{1}{18}\lambda^2+\frac{7}{18}\lambda+\frac{4}{15}.
\end{equation}
The idea is to find a ``simple'', provably close approximation to $r_n$ in $\Hb$, that 
converges to 1 for any fixed $\lambda$, which would then imply~\eqref{limit1_for_rn}. 

We use the quasi-solution approach, initially developed for ordinary differential equations 
in \cite{CostinHuangSchlag12, CostinHuangTanveer14}, which we here, in a sense, extend to difference 
equations of type~\eqref{recurrence_for_rn}. Namely, as a quasi-solution to~\eqref{recurrence_for_rn}
we define
\begin{equation}\label{quasi_solution}
\tilde{r}_n(\lambda)=\frac{\lambda^2}{4n^2+31n+43}+\frac{\lambda}{n+4}+\frac{n+2}{n+4}.
\end{equation}
Of course, the choice is not arbitrary, and in \S\ref{QuasiExplanation} 
we describe in some detail how to obtain such an approximate solution. 
The quasi-solution $\tilde{r}_n$ turns out to be a good approximation to $r_n$ in the whole of $\Hb$.

\begin{lem}
\label{analyticity}
$r_1$ and $(\tilde{r}_n)^{-1}$ for $n\geq1$, are analytic in $\Hb$.
\end{lem}

\begin{proof}
From~\eqref{recurrence_for_rn} and~\eqref{r_-1} we compute
\small\begin{equation*}
r_1(\lambda)= \frac{1}{78}\frac{25\lambda^6+825\lambda^5+10945\lambda^4+69735\lambda^3+207694
\lambda^2+260856\lambda+96192}{25\lambda^4+450\lambda^3+2735\lambda^2+5070\lambda+2016}.
\end{equation*}\normalsize
The denominator of $r_1$ and the polynomials $\tilde{r}_n(\lambda)$ for $n\geq1$ 
are Hurwitz-stable i.e., all of their zeros are in the (open) left half-plane, which can be straightforwardly 
checked by, say, the Routh-Hurwitz criterion or its reformulation by Wall (see \cite{Wall45} or 
\S\ref{Hurwitz_Stability})\footnote{There are, of course, elementary ways of proving this claim. 
However, the suggested approach is more general.}. The conclusion follows.
\end{proof}

Now, let
\begin{equation}\label{delta_definition} 
\delta_n=\frac{r_n}{\tilde{r}_n}-1.
\end{equation}
Substitution of~\eqref{delta_definition} into~\eqref{recurrence_for_rn} 
leads to the following recurrence relation for $\delta_n$,
\begin{equation}\label{delta_recurrence} 
\delta_{n+1}=\varepsilon_n+C_n\frac{\delta_n}{1+\delta_n},
\end{equation}
where
\begin{equation}\label{epsilon_and_C} 
\varepsilon_n=\frac{-A_n\tilde{r}_n-B_n}{\tilde{r}_n\tilde{r}_{n+1}}-1 \quad \text{and} 
\quad C_n=\frac{B_n}{\tilde{r}_n\tilde{r}_{n+1}}.
\end{equation}

\begin{lem}
\label{estimates}
The following estimates hold in $\Hb$,
\begin{gather}
|\delta_1|\leq\frac{1}{4}, \label{d1_estimate} \\
|\varepsilon_n|\leq\frac{1}{20},\quad n\geq1, \label{en_estimate} \\
|C_n|\leq\frac{3}{5}, \quad n\geq1. \label{cn_estimate} 
\end{gather}
\end{lem}

\begin{proof} 
\par The method of proof is the same for all three quantities, so we illustrate it 
only on $C_n$. \par Lemma \ref{analyticity} and~\eqref{epsilon_and_C} 
imply that $C_n$ is analytic in $\Hb$. 
Also, being a rational function, $C_n$ is evidently polynomially 
bounded in $\Hb$. Hence, according to the 
Phragm\'{e}n-Lindel\"{o}f principle\footnote{We use the sectorial 
formulation of this principle, see, for example, \cite{Titchmarsh58}, p. 177.}, 
it suffices to prove that~\eqref{cn_estimate} holds on the imaginary line. 
To that end, we first bring $C_{n+1}(\lambda)$ to the form of the ratio of two 
polynomials $P_1(n,\lambda)$ and $P_2(n,\lambda)$\footnote{For all three quantities, 
straightforward calculations would lead to the form that we used. However, to prevent 
possible ambiguity, in \S\ref{polynomials} we give the explicit form (as a ratio of 
polynomials) for all three quantities.}. Then, for $t$ real, $|C_{n+1}(it)|^2$ is equal 
to the quotient of two polynomials, $Q_1(n,t^2)=|P_1(n,it)|^2$ 
and $Q_2(n,t^2)=|P_2(n,it)|^2$. In order to show that $|C_{n+1}(it)|\leq3/5$, for all 
real $t$ and $n\geq0$, all we need is to show 
that $|C_{n+1}(it)|^2=Q_1(n,t^2)/Q_2(n,t^2)\leq9/25$, or 
equivalently $9/25\cdot Q_2-Q_1\geq0$. Using elementary calculations, 
we see that $9/25\cdot Q_2-Q_1$ has manifestly positive coefficients, and 
the variable $t$ appears with even powers only. Thus,~\eqref{cn_estimate} holds on 
the whole imaginary line, and the result follows.
\end{proof}

\begin{proof}[Proof of the Theorem \ref{main_theorem}]
From~\eqref{delta_recurrence} and Lemma \ref{estimates}, a simple inductive argument implies that 
\begin{equation}\label{delta_estimate}
|\delta_n|\leq\frac{1}{4}, \quad \text{for all } n\geq1 \text{, and } \lambda\in\Hb.
\end{equation}
Since for any fixed $\lambda$, $\lim_{n\rightarrow \infty} \tilde{r}_n(\lambda) = 1$,
~\eqref{delta_definition} and~\eqref{delta_estimate} exclude the possibility of~\eqref{limit2_for_rn}. 
Hence,~\eqref{limit1_for_rn} holds in $\Hb$, and the claim follows.
\end{proof}

\section{Appendix}

\subsection{Description of  how to obtain a quasi-solution}\label{QuasiExplanation}
First, the minimax polynomial approximation\footnote{The minimax polynomial approximation of degree 
$n$ to a continuous function $f$ on a given finite interval $[a,b]$ is defined to be the best approximation, 
among the polynomials of degree $n$, to $f$ in the uniform sense on $[a,b]$. For the proof of existence 
and uniqueness of this approximation and an algorithm to obtain it, see \cite{Phillips03}, \S 2.4.} of 
degree two to $r_n$ over an interval $[0,10]$ is found, where $n$ ranges from 0 to 20. Then, appropriate 
rational functions in $n$ are fitted to the coefficients of the approximation polynomials. \par
We should point out that interval of polynomial approximation and the range of values of $n$ can vary, 
and the ones from the description are just our choice. We choose quadratic polynomial approximations 
due to the fact that $r_n$ is a ratio of two polynomials whose degrees differ by two.

\subsection{Wall's criterion for Hurwitz-stability}
\label{Hurwitz_Stability}
Let $P(z)=z^n+a_1z^{n-1}+\cdots+a_n$ be a polynomial with real coefficients, and let 
$Q(z)=a_1z^{n-1}+a_3z^{n-3}+\cdots$ be the polynomial that contains exactly those terms of $P(z)$ 
that have odd-indexed coefficient. Then all the zeros of $P(z)$ have negative real parts if and only if the 
quotient $Q(z)/P(z)$ can be represented in a finite continued fraction form
$$1/(a_1+1/(a_2+1/(a_3+ \ldots +1/a_n)\dots),$$ where $a_1=c_1z+1$, $a_2=c_2z,$ $\dots,$ $a_n=c_nz$, 
and the coefficients $c_1$, $c_2$, $\dots$, $c_n$ are all positive.

In our case, for the denominator of $r_1$, the 
coefficients $c_i$ are $c_1=\mbox{\small$1/18$}$, $c_2=\mbox{\small$135/736$}$, 
$c_3=\mbox{\small$33856/64863$}$ and $c_4=\mbox{\small$36035/15456$}$, 
and for $\tilde{r}_n$, $c_1=\mbox{\small$(n+4)/(4n^2+31n+43)$}$, and $c_2=\mbox{\small$1/(n+2)$}$. 
\normalsize

\subsection{Detailed expressions for $C_n$, $\varepsilon_n$ and $\delta_1$}
\label{polynomials}
We give details of these quantities in order to fully clarify the notations. 
We have 
\[ C_{n+1}=P_1(n,\lambda)/P_2(n,\lambda), \]
where 
\begin{align*}
P_1(n,\lambda)=&-3(n+5)(n+6)(4n^2+39n+78)(4n^2+47n+121) \\
&\times [\lambda^2+(4n+11)\lambda+4n^2+22n+28] 
\end{align*}
and
\begin{align*} 
P_2(n,\lambda)=&10(2n^2+23n+60)[(n+5)\lambda^2+(4n^2+39n+78)(\lambda+n+3)] \\
&\times [(n+6)\lambda^2+(4n^2+47n+121)(\lambda+n+4)], 
\end{align*}
respectively. Furthermore, $\varepsilon_{n+1}=P_3(n,\lambda)/P_2(n,\lambda)$, where 
\begin{align*}
P_3(n,\lambda)=&5(n+1)(n+5)(n+6)\lambda^4 \\
&-5(8n^4+158n^3+1095n^2+3171n+3162)\lambda^3 \\
&-(112n^5+2364n^4+17243n^3
+48805n^2+33244n-36060)\lambda^2 \\
&-4(4n^2+39n+78)(4n^2+47n+121) \\
&\times [(3n^2+5n-3)\lambda-4n^2-3n+36]. \end{align*} 
Finally, \begin{equation*}
\delta_1=\frac{-5\lambda^2(15\lambda^3-20\lambda^2-939\lambda+1412)-36(1093\lambda-256)}
{(5\lambda^2+78\lambda+234)(25\lambda^4+450\lambda^3+2735\lambda^2+5070\lambda+2016)}.
\end{equation*}

\bibliography{references}
\bibliographystyle{plain}

\end{document}